\documentclass[a4paper,preprint,12pt]{elsarticle}

\usepackage{amssymb}

\journal{Information Processing Letter}

\newtheorem{theorem}{Theorem}
\newtheorem{corollary}[theorem]{Corollary}

\newtheorem{lem}[theorem]{Lemma}

\newtheorem{prop}[theorem]{Proposition}
\newproof{proof}{Proof}

\begin{document}

\begin{frontmatter}



\title{Covering graphs with convex sets and partitioning graphs into convex sets\tnoteref{t1}}

\tnotetext[t1]{L.M.~Gonz\'alez, L.N.~Grippo, and M.D.~Safe were partially supported by PIO CONICET UNGS-144-20140100011-CO and UNS Grant PGI 24/L103. M.D. Safe acknowledges partial support of MATH-AmSud 18-MATH-01 and the French-Argentinian International Laboratory SINFIN. V.F.~dos Santos was partially supported by grants from CAPES, CNPq, and FAPEMIG.}



\author[UNGS]{Luc\'ia M. Gonz\'alez}
\ead{lgonzale@ungs.edu.ar}
\author[UNGS]{Luciano N. Grippo}
\ead{lgrippo@ungs.edu.ar}
\author[DM,INMABB]{Mart\'in D. Safe}
\ead{msafe@uns.edu.ar}
\author[UFMG]{Vin\'icius F. dos Santos}
\ead{viniciussantos@dcc.ufmg.br}

\address[UNGS]{Instituto de Ciencias, Universidad Nacional de General Sarmiento, Los Polvorines, Buenos Aires, Argentina}
\address[DM]{Departamento de Matem\'atica, Universidad Nacional del Sur (UNS), Bah\'ia Blanca, Buenos Aires, Argentina}
\address[INMABB]{INMABB, Universidad Nacional del Sur (UNS)-CONICET, Bah\'ia Blanca, Buenos Aires, Argentina}
\address[UFMG]{Departamento de Ci\^{e}ncia da Computa\c{c}\~{a}o, Universidade Federal de Minas Gerais, Belo Horizonte, Minas Gerais, Brazil}

\begin{abstract}
We present some complexity results concerning the problems of covering a graph with $p$ convex sets and of partitioning a graph into $p$ convex sets. The following convexities are considered: digital convexity, monophonic convexity, $P_3$-convexity, and $P_3^*$-convexity. 
\end{abstract}

\begin{keyword}

convex \sep $p$-cover \sep convex $p$-partition \sep digital convexity \sep monophonic convexity \sep $P_3$-convexity \sep $P_3^*$-convexity




\end{keyword}

\end{frontmatter}



\section{Introduction}\label{intro}
A \emph{convexity} of a graph $G$ is a pair $(V(G),\mathcal{C})$ where $\mathcal{C}$ is a family of subsets of $V(G)$ satisfying all the following conditions: $ \emptyset \in \mathcal{C}$, $V(G)\in \mathcal{C}$, and $\mathcal{C}$ is closed under intersections (i.e., $V_1\cap V_2\in \mathcal{C}$ for each $V_1,V_2\in \mathcal{C}$). Each set of the family $\mathcal{C}$ is called \emph{$\mathcal{C}$-convex}. Distinct convexities of a graph have been widely studied in the last years. Several articles can be found in the literature dealing with algorithmic and complexity issues of parameters related to different kind of convexities. Interesting enough is to compare the behavior of these parameters under different convexities from a computational complexity perspective. For an introduction to the different parameters studied in the literature related to a convexity of a graph, see e.g.~\cite{Douchet1988}.

A family $\{V_1,\ldots,V_p\}$ of $p$ $\mathcal{C}$-convex sets, each of which different from $V(G)$, is a \emph{$\mathcal{C}$-convex $p$-cover} if $V_1\cup\cdots\cup V_p=V(G)$ and a \emph{$\mathcal{C}$-convex $p$-partition} if, in addition, $V_i\cap V_j=\emptyset$ for each $i,j\in\{1,\ldots,p\}$ with $i\neq j$. In this work, we consider two associated decision problems: in one case, the input is just a graph $G$ and, in the other case, the input is a graph $G$ and an integer $p\ge 2$. In both cases, the decision problem is: `Does the graph $G$ have a $\mathcal{C}$-convex $p$-cover (resp.\ $\mathcal{C}$-convex $p$-partition)?'.

All graphs in this work are finite, undirected, without no loops and no multiple edges. Let $G$ be a graph. We denote by $V(G)$ and $E(G)$ the vertex set and the edge set of $G$, respectively. The neighborhood of a vertex $v$ of $G$ is denoted by $N_G(v)$, and $N_G[v]$ stands for the set $N_G(v)\cup\{v\}$. If $X$ is a set, $\vert X\vert$ denotes its cardinality. The \emph{degree} of a vertex $v$ of $G$ is $\vert N_G(v)\vert$. If $S$ is a subset of $V(G)$, then the set of those vertices of $G$ with at least one neighbor in $S$ is denoted by $N_G(S)$, and $N_G[S]$ stands for $N_G(S)\cup S$. The \emph{length} of a path is its number of edges. The \emph{distance} between two vertices $u$ and $v$ of $G$, denoted by $d_G(u,v)$, is the minimum length of a path having $u$ and $v$ as end-vertices. The \emph{diameter} of a graph is the maximum distance between two of its vertices. A path is \emph{induced} if there is no edge joining two nonconsecutive vertices of the path. We denote by $P_3$ the induced path on three vertices. We denote the \emph{complement graph} of $G$ by $\overline G$. Graph $G$ is \emph{co-bipartite} if $\overline G$ is bipartite. For any other graph-theoretic notions not given here, see~\cite{West01}.

Let $\mathcal{P}$ be a set of paths in $G$ and let $S\subseteq V(G)$. If $u$ and $v$ are two vertices of $G$, then the \emph{$\mathcal{P}$-interval} $J_{\mathcal{P}}[u,v]$ is the set of all vertices lying in some path $P\in \mathcal{P}$ having $u$ and $v$ as its end-vertices. Let $J_{\mathcal P}[S]=\bigcup_{u,v\in S}{J_{\mathcal P}[u,v]}$. Let $\mathcal{C}$ be the family of all sets $S$ of vertices of $G$ such that for each path $P\in \mathcal{P}$ whose end-vertices belong to $S$, every vertex of $P$ also belongs to $S$; i.e., $ \mathcal{C}$ consists in those subsets $S$ of $V(G)$ such that $J_{\mathcal{P}}[S]=S$. It is easy to show that $(V(G),\mathcal{C})$ is a convexity of $G$ and $\mathcal{C}$ is called the \emph{path convexity generated by $\mathcal P$}. Some of the most studied path convexities are the \emph{geodesic convexity}, the \emph{monophonic convexity}, the \emph{$P_3$-convexity}, and the \emph{$P_3^*$-convexity}, which are the convexities whose convex sets are generated by the set of all minimum paths, the set of all induced paths, the set of all paths of length three, and the set of all induced paths of length three of the graph, respectively.

A set $S$ of vertices of a graph $G$ is \emph{digitally convex} if, for each vertex $v$ of $G$, $N_G[v]\subseteq N_G[S]$ implies $v\in S$.  The \emph{digital convexity} of a graph $G$ is the pair $(V(G),\mathcal{C})$ where $\mathcal{C}$ is the set of all digitally convex sets of $G$. This convexity was introduced in~\cite{PR-digitalconvexity-1966} as a tool to filter digital images.

The problem of deciding whether $G$ has a $\mathcal{C}$-convex $p$-partition was introduced by Artigas et al.~\cite{paper2}. They proved that the problem is NP-complete for each fixed integer $p\geq 2$, under the geodesic convexity. Centeno et al.~\cite{paper1} proved that the problem is also NP-complete under the $P_3$-convexity if a graph $G$ and an integer $p\geq 2$ are given as input. The problem of deciding whether a graph has a $\mathcal{C}$-convex $p$-cover was introduced in~\cite{paper3} also by Artigas et al. This problem is NP-complete under the geodesic convexity for each fixed integer $p\geq 2$~\cite{paper3,paper4}.

This article is organized as follows. In Section~\ref{sec:digital convexity}, we deal with the problems of covering with convex sets and of partitioning into convex sets a graph under the digital convexity. In Section~\ref{sec:P_3-convexity}, we present results in connection with the problem of partitioning a graph into convex sets under the $P_3$-convexity, as well as the $P_3^*$-convexity, and we also consider the problem of covering a graph with $P_3$-convex sets. In Section~\ref{sec:monophonic}, we address these problems under the monophonic convexity.

\section{Digital convexity}\label{sec:digital convexity}

We will call \emph{d-convex} to any convex set under the digital convexity. Our first result characterizes d-convex sets as the complements of the closed neighborhoods of sets of vertices.

\begin{prop}\label{prop:1} 
Let $G$ be a graph, $S \subseteq V(G)$, and $W = V(G)\setminus N_G[S]$. For each vertex $v$ of $G$, $N_G[v]\subseteq N_G[S]$ if and only if $v \notin N_G[W]$.
\end{prop}
\begin{proof} In fact, $N_G[v] \subseteq N_G[S]$ is equivalent to $N_G[v] \cap W = \emptyset$ which, in turn, is equivalent to $v \notin N_G[W]$.
\end{proof}

\begin{lem}\label{lem:complementofclosedneighborhood}
A set of vertices $S$ of a graph $G$ is d-convex if and only if $S=V(G)\setminus N_G[W]$ for some set $W \subseteq V(G)$.
\end{lem}

\begin{proof} Suppose that $S$ is a d-convex set of $G$ and let $W=V(G)\setminus N_G[S]$. We claim that $S=V(G)\setminus N_G[W]$. On the one hand, if $v\in S$, then $N_G[v]\subseteq N_G[S]$ and, by virtue of Proposition~\ref{prop:1}, $v\notin N_G[W]$. On the other hand, if $v\notin N_G[W]$, then, by Proposition~\ref{prop:1}, $N_G[v]\subseteq N_G[S]$ and, since $S$ is d-convex, $v\in S$. We conclude that $S=V(G)\setminus N_G[W]$, as desired.

Conversely, let $W\subseteq V(G)$ and $S=V(G)\setminus N_G[W]$. Let $v\in V(G)$ such that $N_G[v]\subseteq N_G[S]$. 
Notice that $N_G[S]\cap W=\emptyset$ (in fact, if there were some vertex $w\in W\cap N_G[S]$, then there would be some vertex $s\in N_G[W]\cap S$, a contradiction). Thus, $N_G[v]\cap W=\emptyset$ or, equivalently, $\{v\}\cap N_G[W]=\emptyset$. Hence, $v\in V(G)\setminus N_G[W]=S$. This proves that $S$ is a d-convex set of $G$.\end{proof}


A \emph{total dominating set} of a graph $G$ is a set $S$ of vertices of $G$ such each vertex of $G$ has at least one neighbor in $S$. Our next result shows that the existence of a covering with few digitally convex sets is equivalent to the existence of sufficiently small total dominating sets in the complement.

\begin{theorem}\label{thm:dominating set}
A graph $G$ has a d-convex $p$-cover if and only if $\overline G$ has a total dominating set of cardinality at most $p$.
\end{theorem}

\begin{proof}
Let $X=\{w_1,w_2,\ldots,w_k\}\subseteq V(G)$ be a total dominating set of $\overline G$ with $k\leq p$. By virtue of Lemma~\ref{lem:complementofclosedneighborhood}, $V_i= V(G)\setminus N_G[w_i]$ is a d-convex set for each $i\in\{1,\ldots,k\}$. We claim that $\{V_1,V_2,\ldots,V_k\}$ is a cover of $V(G)$. Since $X$ is a total dominating set of $\overline G$, for each $v$ in $ V(\overline{G})\ ({}=V(G))$, there exists some $i\in\{1,\ldots,k\}$ such that $v\in N_{\overline G}(w_i)=V(G)\setminus N_G[w_i]=V_i$. Therefore, $V(G)=\bigcup_{i=1}^k V_i$.

Conversely, let $\{V_1,\ldots,V_p\}$ be a d-convex $p$-cover of $G$. Hence, by Lemma \ref{lem:complementofclosedneighborhood}, $V_i=V(G)\setminus N_G[W_i]$ for some $W_i\subseteq V(G)$. Choose a vertex $w_i\in W_i$ for each $i\in\{1,\ldots,p\}$. Notice that it may happen that $w_i=w_j$ even if $i\neq j$. Let $W=\{w_1,\ldots,w_p\}$. By construction, $\vert W\vert\leq p$. As $\{V_1,\ldots,V_p\}$ is a cover of $G$, for each $v\in V(\overline{G})$ there is some $i\in\{1,\ldots,p\}$ such that $v\in V_i=V(G)\setminus N_G[W_i]\subseteq V(G)\setminus N_G[w_i]=N_{\overline G}(w_i)$. This proves that $W$ is a total dominating set of $\overline G$ with at most $p$ vertices.\end{proof}

Since the problem of deciding, given a graph $G$ and a positive integer $p$, whether $G$ has a total dominating set of size at most $p$ is NP-complete~\cite{td}, Theorem~\ref{thm:dominating set} implies the following.

\begin{corollary}\label{cor:p-cover_d-convex}
It is NP-complete to decide, given a graph $G$ and an integer $p\geq 2$, whether $G$ has a d-convex $p$-cover.
\end{corollary}

Notice that Corollary~\ref{cor:p-cover_d-convex} proves the NP-completeness of deciding whether a graph $G$ has d-convex $p$-cover when $p$ is part of the input. However, the complexity for the case in which $p$ is a fixed integer is unknown.

We were not able to determine the computational complexity of the problem of deciding, given a graph $G$ and an integer $p$, whether the graph $G$ has a d-convex $p$-partition. Below, we present a result that shows that if the graph $G$ is assumed bipartite and we fix $p=2$, then the problem becomes polynomial-time solvable.
We first prove the following lemma.

\begin{lem}\label{lem:diameter}
Let $G$ be a connected graph. If $G$ has a d-convex $2$-partition, then $G$ has diameter at least 3.
\end{lem}
\begin{proof}
Let $\{V_1,V_2\}$ be a d-convex $2$-partition of a connected graph $G$. Since $G$ is connected, there exists $u_i\in V_i$ for each $i\in\{1,2\}$ such that $u_1$ is adjacent to $u_2$. Since $u_i\in N_G[V_{i+1}]$ (where sums should be considered modulo 2) and $V_{i+1}$ is a $d$-convex set, there exists some vertex $v_i$ adjacent to $u_i$ such that $v_i\notin N_G[V_{i+1}]$. By construction, $v_1,u_1,u_2,v_2$ is an induced path of length three of $G$. Consequently, the diameter of $G$ is at least $3$.
\end{proof}

We are now ready to prove the following result.

\begin{theorem}\label{thm:2-partition_d-convex}
A bipartite graph $G$ has a d-convex $2$-partition if and only if $G$ has diameter at least 3. Moreover, if such a partition exists, it can be found in polynomial time.
\end{theorem}

\begin{proof} If $G$ is disconnected and $\mathcal{C}$ is any connected component of $G$, then $\{\mathcal{C},V(G)-\mathcal{C}\}$ is clearly a d-convex $2$-partition of $G$ that, in addition, can be built in polynomial time. Hence we assume, without loss of generality, that $G$ is connected. We have already proved in Lemma~\ref{lem:diameter} that if $G$ is a connected graph having a d-convex $2$-partition, then $G$ has diameter at least $3$. Conversely, assume now that $G$ is a connected bipartite graph with bipartition $\{X,Y\}$ and diameter at least 3. Two vertices $u,v\in V(G)$ such that $d_G(u,v)=3$ and the two sets $V_1=N_G[u]\cup\{x\in X\colon\,N_G(x)\subseteq N_G[u]\}$ and $V_2=V(G)\setminus V_1$ can be computed polynomial time. It is not hard to see that $\{V_1,V_2\}$ is a d-convex $2$-partition of $G$.
\end{proof}

\section{Convexities generated by paths of length three}\label{sec:P_3-convexity}
We1 call \emph{$P_3$-convex} (resp.\ \emph{$P_3^*$-convex}) to any convex set under the $P_3$-convexity (resp.\ $P_3^*$-convexity). Let $G$ be a bipartite graph with a bipartition $\{X,Y\}$ and let $p$ be an integer such that $p\geq 2$. We construct a bipartite graph $G'$ as follows: we take a copy of $G$ and a complete bipartite graph $K_{r,r}$ where $r=\max\{p+2,\vert X\vert,\vert Y\vert\}$, we add an edge connecting each vertex in $X$ with a vertex of one of the partite sets of $K_{r,r}$, and we add an edge connecting each vertex in $Y$ with a vertex of the other partite set of $K_{r,r}$, so that there are no two vertices in the copy of $G$ in $G'$ adjacent to the same vertex in the complete bipartite graph~$K_{r,r}$.

A \emph{cut} of a graph $G$ is a partition of $V(G)$ into two sets $X$ and $Y$, denoted by $(X,Y)$. The set of all edges having one endpoint in $X$ and the other one in $Y$ is called the \emph{edge cut} of the cut $(X,Y)$. A \emph{matching cut set} is an edge cut that is a (possible empty) matching (i.e., no two edges of the edge cut share an endpoint). 

It can be easily proved that the problem of deciding whether a graph has a $P_3$-convex $2$-partition and the problem of deciding whether a graph has matching cut set are equivalent.

\begin{lem} \label{lem:2partitionP3}
A graph $G$ has a $P_3$-convex $2$-partition if and only if $G$ has a matching cut set.
\end{lem}
%

\begin{lem} \label{lem:p-convex_then_p+1convex}
Let $p$ be an integer such that $p\geq 2$. A graph $G$ has a $P_3$-convex $p$-partition if and only if $G'$ has a $P_3$-convex $(p+1)$-partition.
\end{lem}
\begin{proof}
If $G$ has a $P_3$-convex $p$-partition $\{V_1,\ldots,V_p\}$, then $G'$ has a $P_3$-convex $(p+1)$-partition $\{V_1,\ldots,V_p,V_{p+1}\}$, where $V_{p+1}$ is the set of vertices of the copy of $K_{r,r}$ in $G'$.

Conversely, suppose that $G'$ has a $P_3$-convex $(p+1)$-partition $\{V_1,\ldots,V_p,\linebreak V_{p+1}\}$. Let $V$ be the vertex set of the copy of $G$ in $G'$ and let $V'$ be the vertex set of the copy of $K_{r,r}$ in $G'$. We claim that $V'\subseteq V_i$ for some $i\in\{1,\ldots,p,p+1\}$. Indeed, since $r\ge p+2$, there exist at least two nonadjacent vertices $x,y\in V'\cap V_i$ for some $i\in\{1,\ldots,p,p+1\}$ and thus $V'\subseteq V_i$ (because $V_i$ is $P_3$-convex). We assume, without losing generality, that $i=p+1$; i.e., $V'\subseteq V_{p+1}$. Let $W=V_{p+1}\cap V$. Arguing towards a contradiction, suppose that $W\neq\emptyset$. Notice that if no vertex in $W$ has a neighbor in $V_i$ for some $i\in\{1,\ldots,p\}$, then $\{V_1,\ldots,V_{p-1},V_p\cup W\}$ is a $P_3$-convex $p$-partition of the copy $G$ in $G$' and, in particular, a $P_3$-convex $p$-partition of $G$, as desired. Hence, we assume, without loss of generality, that there is some vertex $a\in W$ adjacent to some vertex $b\in V_i$ for some $i\in\{1,\ldots,p\}$. By the construction of $G'$, there exists a vertex $c\in V'\subseteq V_{p+1}$ which is adjacent to $b$ and nonadjacent to $a$. Consequently, $a,b,c$ is a path on three vertices of $G'$ such that $a,c\in V_{p+1}$ but $b\notin V_{p+1}$, contradicting the fact that $V_{p+1}$ is a $P_3$-convex set of $G'$. This contradiction arose from assuming that $W\neq\emptyset$. Hence, $W=\emptyset$ and thus $V_{p+1}=V'$. Therefore, $\{V_1,\ldots,V_p\}$ is a $P_3$-convex $p$-partition of the copy of $G$ in $G'$ and, in particular, a $P_3$-convex $p$-partition of $G$.
\end{proof}

Since the matching cut set problem is NP-complete even when the input is a bipartite graph~\cite{LR-2003}, by virtue of Lemma~\ref{lem:p-convex_then_p+1convex}, we obtain the following theorem.
\begin{theorem}\label{cor:p-partition_p3-convex_in_bipartite} For each fixed integer $p\geq 2$, it is NP-complete to decide, given a graph $G$, whether $G$ has a $P_3$-convex $p$-partition, even if $G$ is a bipartite graph.
\end{theorem}
Notice that, in a bipartite graph, the family of $P_3$-convex sets in $G'$ coincides with the family of $P_3^*$-convex sets in $G'$. Hence, Corollary~\ref{cor:p-partition_p3-convex_in_bipartite} still holds if `$P_3$-convex' is replaced by `$P_3^*$-convex'.

A \emph{stable set} is a set of pairwise nonadjacent vertices. A \emph{clique} is a set of pairwise adjacent vertices. A \emph{split graph} is a graph whose vertex set can be partitioned into an independents set $S$ and a clique $K$; the pair $(K,S)$ is called a \emph{split partition}.
\begin{lem}\label{lem:split_graph-p3-convexity}
Let $G$ be a split graph with split partition $(K,S)$ and $p$ be an integer such that $p\ge 2$. If every vertex $s\in S$ has degree at least two, then $G$ has a $P_3$-convex $p$-partition if and only if $G$ has a $P_3$-convex $p$-cover.
\end{lem}

\begin{proof}
The `only if' part is clear. Assume that $G$ has a $P_3$-convex $p$-cover $\{V_1,V_2,\ldots,V_p\}$. Notice that, for each $i\in\{1,\ldots,p\}$, $\vert V_i\cap \mathcal{C}\vert \leq 1$ because otherwise every vertex $x\in \mathcal{C}$ would belong to $V_i$ and, since every vertex of $S$ is adjacent to at least two vertices of $\mathcal{C}$, $y\in V_i$ for every vertex $y\in S$, contradicting $V_i\neq V(G)$. Besides, if $x\in \mathcal{C}$ and $y\in S$ are adjacent, then there is no $i\in\{1,\ldots,p\}$ such that $x,y\in V_i$; otherwise, any other neighbor $z$ of $y$ in $\mathcal{C}$ would belong to $V_i$ and so $x,z\in V_i\cap \mathcal{C}$. We have proved that $V_i$ is an independent set for each $i\in\{1,\ldots,p\}$. Let $V'_1=V_1$ and let $V'_{i+1}=V_{i+1}\setminus(\bigcup_{j=1}^i V_j)$ for each $j\in\{2,\ldots,p\}$. By construction, $\{V'_1,\ldots,V'_p\}$ is a partition. We claim that, for each $i\in\{1.\ldots.p\}$, $V'_i$ is a $P_3$-convex set. Arguing towards a contradiction, suppose that there exists some $i\in\{1,\ldots,p\}$ and some vertex $x\in V(G)\setminus V'_i$ such that $\{u,v\}\subseteq N_G(x)\cap V'_i$. In particular, $u,v\in V_i$ and, since $V_i$ is $P_3$-convex, $x\in V_i$, contradicting the fact that $V_i$ is an independent set. This contradiction proves the claim. Hence, $\{V_1',V_2',\ldots,V_p'\}$ is a $P_3$-convex $p$-partition of $G$.\end{proof}


In~\cite{paper1}, it is proved that it is NP-complete to decide whether a split graph $G$ with split partition $(K,S)$ has a $P_3$-convex $p$-partition when $p$ is part of the input. Looking carefully at the proof given in~\cite{paper1}, one can readily verify that it is still valid if every vertex in $S$ has degree at least two.

\begin{theorem}[{\cite[Theorem 2.1]{paper1}}]\label{thm:p3-convex-partition}
The problem of deciding, given a split graph $G$ with a split partition $(K,S)$, where $G$ has a $P_3$-convex $p$-partition is NP-complete, even if each vertex in $S$ has at least two neighbors.\end{theorem}

Therefore, in virtue of Lemma~\ref{lem:split_graph-p3-convexity}, the result below follows.

\begin{corollary}\label{cor:P3-convex p-partition}
It is NP-complete to decide, given a graph $G$ and an integer $p\ge 2$, whether $G$ has a $P_3$-convex $p$-cover, even if $G$ is a split graph.
\end{corollary}

\section{Monophonic convexity}\label{sec:monophonic}

We will call m-convex to any convex set under the monophonic convexity. We will prove that it is NP-complete to decide, given a graph $G$ and an integer $p\ge 3$, whether $G$ has an m-convex $p$-partition, by adapting the proof of~\cite[Theorem 1]{paper3} for the analogous result for the geodesic convexity. 
If $p$ is a positive integer, \textsc{Clique $p$-Partition} is the problem of deciding, given a graph $G$, whether the vertex set of $G$ can be partitioned into $p$ cliques of $G$. It is well known that this problem is NP-complete for all $p\geq 3$.


\begin{theorem}\label{thm:p-cover_m-convex}
It is NP-complete to decide, given a graph $G$ and an integer $p\ge 3$, whether $G$ has an m-convex $p$-cover.
\end{theorem}

\begin{proof}
Deciding if a set $S\subseteq V(G)$ is m-convex can be done in polynomial time~\cite{paper5} and thus the problem of deciding if a graph $G$ has an m-convex $p$-cover belongs to $NP$.
Let $G$ be an instance of the \textsc{Clique $p$-Partition} problem. We assume, without loss of generality, that $\vert V(G)\vert\ge 2$ and $G$ is not a complete graph. We construct a graph $G'$ whose vertex set is $V(G)\cup\{u,v\}$ and $E(G')=E(G)\cup \{ux,vx\colon\,x\in V(G)\}$. The proof follows in the same way as the proof of the NP-completeness of Theorem~1 in~\cite[Theorem~1]{paper3}. It is not hard to prove that any proper m-convex set of $G'$ is a clique. From this assertion, it follows that $G'$ has an m-convex $p$-partition if and only if $G$ has an $\ell$-clique partition for some integer $\ell$ such that $p-2\le\ell\le p$.
\end{proof}

Next, we will prove that deciding whether a graph $G$ has an m-convex $p$-cover becomes polynomial-time solvable for $p=2$. A \emph{clique separator} of a connected graph $G$ is a clique $\mathcal{C}$ such that $G-\mathcal{C}$ is disconnected.

\begin{lem}[{\cite{paper5}}]\label{lem:m-convex 2-cover}
Let $G$ be a connected graph, $\mathcal{C}$ a clique separator of $G$, and $S$ the union of the vertex sets of some of the connected components of $G-\mathcal{C}$. Then $S\cup \mathcal{C}$ is an m-convex set of $G$.
\end{lem}

\begin{corollary}\label{cor:m-convex 2-conver clique separator}
If $G$ is a connected graph having a clique separator, then $G$ has an m-convex 2-cover.
\end{corollary}

\begin{proof}
If $\mathcal{C}$ is a clique separator of $G$ and $G_1$ is any connected component of $G-\mathcal{C}$, then, by virtue of Lemma~\ref{lem:m-convex 2-cover}, $\{V(G_1)\cup \mathcal{C},(V(G)\setminus V(G_1))\cup \mathcal{C}\}$ is an m-convex 2-cover of $G$.
\end{proof}

Let $G$ be a graph. The \emph{m-convex hull} of a set $S\subseteq V(G)$ is the inclusion-wise minimal m-convex set $M$ of $G$ containing $S$ or, equivalently, $M=J_{\mathcal P}[S]$ where $\mathcal P$ is the set of all induced paths of $G$. A set $X$ of vertices of $G$ is an \emph{m-hull set} of $G$ if the m-convex hull of $X$ is $V(G)$.

\begin{lem}[{\cite{paper5}}]\label{lem:convex-mitre}
If $G$ is a connected graph having no clique separator that is not a complete graph, then every pair of nonadjacent vertices is an m-hull set of $G$.
\end{lem}

The following result is an immediate consequence of Lemma~\ref{lem:convex-mitre}.

\begin{corollary}\label{cor: proper m-convex}
If $G$ is a connected graph having no clique separator, then every proper m-convex of $G$ is a clique. Moreover, $G$ has an m-convex $2$-cover if and only if $G$ is a co-bipartite graph.
\end{corollary}



\begin{theorem}\label{thm:m-convex-cover}
It is polynomial-time solvable to decide, given a graph $G$, whether $G$ has an m-convex $2$-cover. 
\end{theorem}

\begin{proof}
In order to decide whether a connected graph $G$ has an m-convex $2$-cover, we first look for a clique separator in $O(nm)$ time~\cite{clique}. If a clique separator is found, then, by Corollary~\ref{cor:m-convex 2-conver clique separator}, $G$ has an m-convex 2-cover. If $G$ has no clique separator, we test if $G$ is co-bipartite, which can be performed in linear time. Consequently, it can be decided in $O(nm)$ time whether or not a given graph $G$ has an m-convex 2-cover.\end{proof}



\section*\refname
\bibliographystyle{elsarticle-num} 
\bibliography{convex}





\end{document}